\documentclass[10pt]{amsart}
\usepackage{mathrsfs}
\usepackage{amsfonts}
\usepackage{amssymb}
\usepackage{amsxtra}

\usepackage{dsfont}

\usepackage[english,polish]{babel}
\newcommand{\pl}[1]{\foreignlanguage{polish}{#1}}

\theoremstyle{plain}
\newtheorem{theorem}{Theorem}
\newtheorem{proposition}{Proposition}[section]

\newtheorem{lemma}[proposition]{Lemma}
\newtheorem*{loglemma}{Logarithmic lemma}
\theoremstyle{definition}

\theoremstyle{remark}
\newtheorem{remark}{Remark}

\newcounter{thm}
\theoremstyle{plain}

\newtheorem{main_theorem}[thm]{Theorem}

\newcommand{\RR}{\mathbb{R}}
\newcommand{\ZZ}{\mathbb{Z}}

\newcommand{\CC}{\mathbb{C}}
\newcommand{\NN}{\mathbb{N}}
\newcommand{\QQ}{\mathbb{Q}}

\newcommand{\seq}[2]{{#1}: {#2}}

\renewcommand{\atop}[2]{\substack{{#1}\\{#2}}}
\newcommand{\norm}[1]{{\lvert #1 \rvert}}

\newcommand{\sprod}[2] {{#1 \cdot #2}}

\newcommand{\abs}[1]{{\lvert {#1} \rvert}}
\newcommand{\sabs}[1]{{\left\lvert {#1} \right\rvert}}
\newcommand{\vnorm}[1]{{\left\lVert {#1} \right\rVert}}

\newcommand{\calP}{\mathcal{P}}

\newcommand{\calF}{\mathcal{F}}

\newcommand{\calM}{\mathcal{M}}

\newcommand{\calI}{\mathcal{I}}

\newcommand{\calS}{\mathcal{S}}
\newcommand{\calO}{\mathcal{O}}
\newcommand{\calN}{\mathcal{N}}
\newcommand{\supp}{\operatornamewithlimits{supp}}

\newcommand{\ind}[1]{{\mathds{1}_{{#1}}}}

\author{Ben Krause}
\address{UCLA Math Sciences Building\\
         Los Angeles CA 90095-1555}
\email{benkrause23@math.ucla.edu}

\author{Mariusz Mirek}
\address{Mariusz Mirek \\
	Universit\"{a}t Bonn \\
	Mathematical Institute\\
	Endenicher Allee 60\\
	D--53115 Bonn \\
	Germany \&
	Instytut Matematyczny\\
	Uniwersytet \pl{Wroc{\lll}awski}\\
	Pl. Grun\-waldzki 2/4\\
	50-384 \pl{Wroc{\lll}aw}\\
	Poland}
\email{mirek@math.uni-bonn.de}

\author{Bartosz Trojan}
\address{
	Bartosz Trojan\\
	Instytut Matematyczny\\
	Uniwersytet \pl{Wroc{\lll}awski}\\
	Pl. Grun\-waldzki 2/4\\
	50-384 \pl{Wroc{\lll}aw}\\
	Poland}
\email{trojan@math.uni.wroc.pl}

\title[Two-parameter inequality]
{Two-parameter version of Bourgain's \\ inequality: Rational frequencies}

\thanks{The authors thank the Hausdorff Research Institute  for
  Mathematics in Bonn for support and hospitality during the Trimester Program
  ``Harmonic Analysis and Partial Differential Equation''. }

\begin{document}
\selectlanguage{english}

\begin{abstract}
	Our aim is to establish the first two-parameter version of  Bourgain's maximal
	logarithmic inequality on $L^2(\mathbb R^2)$ for the rational frequencies. We achieve this by
	introducing a variant of a two-parameter Rademacher--Menschov inequality. The method allows us to
	control an oscillation seminorm as well.
\end{abstract}

\maketitle

\section{Introduction}
Let $A_n=(-2^{-n-1}, 2^{-n-1})$ for $n\in\NN_0=\NN\cup\{0\}$.  Suppose that $\Lambda\subset \RR$ is a finite
set satisfying the following separation condition: for any $\lambda, \lambda' \in \Lambda$,
if $\lambda \neq \lambda'$ then
\begin{equation}
	\label{eq:2}
	\norm{\lambda - \lambda'} \geq 1.
\end{equation}
In \cite{bou}, Bourgain established the following lemma.
\begin{loglemma}
	There exists a constant $C > 0$ such that for each $f \in L^2(\RR)$ we have
	\begin{align}
		\label{eq:26}
		\Big\|\sup_{n\in \NN_0}
		\Big|
		\sum_{\lambda \in \Lambda} \calF^{-1} \big(\ind{A_{n}^\lambda} \calF f \big)
		\Big|
		\Big\|_{L^2}
		\le C\big(\log|\Lambda|\big)^2
		\|f\|_{L^2},
	\end{align}
	where $A_n^\lambda = \lambda + A_n$ and $\calF$ is the Fourier transform operator on $\RR$. Moreover,
	the implied constant  is independent of the cardinality of the set $\Lambda$.
\end{loglemma}
This logarithmic lemma was introduced by Bourgain to reduce some problems in ergodic theory having a number theoretic
nature to questions in harmonic analysis (compare \cite{bou1, bou2} with \cite{bou}). To be more
precise, let $(X, \mathcal{B}, \mu)$ be a $\sigma$-finite measure space and let $T: X \rightarrow X$
be an invertible measure preserving transformation. The classical Birkhoff's theorem (see \cite{birk})
states that for any $f \in L^p(X, \mu)$ with $p\ge1$ the averages
\[
	A_N f(x) := \frac{1}{N}\sum_{n = 0}^{N-1} f(T^n x)
\]
converges $\mu$-almost everywhere. With the aid of the logarithmic
lemma Bourgain proved the pointwise convergence of 
\[ A_N^{\calP} f(x) := \frac{1}{N}\sum_{n = 0}^{N-1} f\big(T^{\calP(n)} x\big) \]
for all $f\in L^p(X, \mu)$ and  $p > 1$; where, $\calP$ is any integer-valued polynomial. The lemma was applied
to the sets
\[
	\mathscr{R}_s
	=
	\big\{ a/q \in [0, 1] \cap \QQ: (a, q)=1, \text{ and } 2^s \le q<2^{s+1} \big\}
\]
giving an acceptable loss with respect to $s$ in \eqref{eq:26} of the order $s^2$ since $\abs{\mathscr{R}_s} \le 4^s$
(see \cite{bou} for more details).

In fact, in \cite{bou} the logarithmic lemma was proven in a much stronger form: for general frequencies without
the separation condition \eqref{eq:2}.  Not long afterwards, it was observed by Lacey (see \cite{lac1}) that
if $\Lambda \subset Q^{-1} \ZZ$ for some $Q \in \NN$ and satisfies separation condition, then
\begin{equation}
	\label{eq:47}
	\Big\|\sup_{n\in \NN_0}
	\Big|
	\sum_{\lambda \in \Lambda} \calF^{-1} \big(\ind{A_{n}^\lambda} \calF f \big)
	\Big|
	\Big\|_{L^2}
	\le C \log\log \big(Q \sqrt{\abs{\Lambda}} \big)
	\|f\|_{L^2}.
\end{equation}
This version of the logarithmic lemma is strictly adjusted to the problems with arithmetic features.
Roughly speaking, when \eqref{eq:47} applied to $\Lambda =  4^{s+1}\cdot\mathscr{R}_s$ with $Q$ equal to the least common multiple of all
$q\in[2^s, 2^{s+1})\cap\NN$,
we obtain a
satisfactory bound as well,  since $Q \le 2^{(s+1)2^{s}}$.

It turned out that the logarithmic lemma is also useful in continuous problems. Especially
important are applications in time-frequency analysis (see e.g.  \cite{DLTT, Lac, Thi}). Recently,
Nazarov, Oberlin and Thiele \cite{NOT} extended Bourgain's inequality  providing $L^p$ bounds for the
$r$-variational counterpart of \eqref{eq:26} (see also \cite{Dem}, \cite{K2} and \cite{Ob}).

In the present article we are concerned with proving a two-parameter variant of Bourgain's inequality
for rational frequencies. From now on $\Lambda$ will be always a subset of $Q_1^{-1}\ZZ \times Q_2^{-1}\ZZ$,
for some $Q_1,Q_2 \in \NN$ satisfying the following two parameter separation condition: for any
$\lambda = (\lambda_1,\lambda_2), \lambda' = (\lambda_1',\lambda_2') \in \Lambda$, if $\lambda \neq \lambda'$ then
\begin{align}
  \label{eq:1}
	\max\{\abs{\lambda_1 - \lambda_1}, \abs{\lambda_2 - \lambda'_2} \}\geq 1.
\end{align}
 One of our main results is the following.
\begin{main_theorem}
	\label{thm:1}
	There is a constant $C > 0$ such that for all $f \in L^2\big(\RR^2\big)$
	\begin{align}
		\label{eq:8}
		\Big\lVert\sup_{n_1, n_2 \in \NN_0}\Big|\sum_{\lambda \in \Lambda}
		\calF^{-1}
		\big(\ind{R_{n_1,n_2}^\lambda} \calF f \big)\Big|
		\Big\rVert_{L^2}
		\leq
		C
		\log \log \big(Q_1 \sqrt{\abs{\Lambda}}\big)
		\cdot \log \log \big(Q_2 \sqrt{\abs{\Lambda}}\big)
		\vnorm{f}_{L^2},
	\end{align}
	where $R_{n_1,n_2}^\lambda = \lambda + A_{n_1} \times A_{n_2}$ and $\mathcal F$ is the Fourier
	transform operator on $\RR^2$. Moreover, the implied constant  is
        independent of $Q_1, Q_2$ and the cardinality of the set $\Lambda$.
\end{main_theorem}
Our motivations to study the bound \eqref{eq:8} lie behind the ongoing project of the second
author with Jim Wright where the authors study $\ell^p\big(\ZZ^3\big)$ boundedness (for $p>1$)
of the following maximal function
\[
	\mathcal Mf(x, y, z)=\sup_{M,
  	N\in\NN}\Big|\frac{1}{MN}\sum_{m=1}^M\sum_{n=1}^Nf\big(x-m, y-n, z-P(m, n)\big)\Big|
\]
where $P$ is an integer-valued polynomial of two variables. Theorem \ref{thm:1} turned out to be
very useful there.

In fact, since these sort of problems find applications in pointwise ergodic theory, we will be interested
in bounding of \eqref{eq:8} for some variant of two-parameter oscillation seminorm rather than the supremum.
Let us recall that in the one-parameter case an oscillation seminorm $\mathcal O$ for a sequence
$(a_n: n\in\NN_0)$ is defined by
\[
	\mathcal O\big(\seq{a_{ n}}{ n \in \NN_0}\big)
	=
	\bigg(\sum_{k\in\NN}
			\sup_{N_k \leq  n \leq N_{k+1}}
			\abs{a_{n} - a_{ N_k}}^2
	\bigg)^{1/2},
\]
for any lacunary sequence $\big(N_k: k\in\NN\big)$. The seminorm $\mathcal O$ is an important object
when  problems concerning pointwise convergence are considered. Indeed, if $\big(a_n(x): n\in\NN_0\big)$
is a sequence of functions such that $\mathcal O\big(\seq{a_{ n}(x)}{ n \in \NN_0}\big)$ is finite for
every lacunary sequence $\big(N_k: k\in\NN\big)$, then the limit $\lim_{n\to\infty} a_{n}(x)$
exists. Thus, we immediately obtain almost everywhere convergence
without relying on a (possibly-unavailable) density argument. Moreover, the oscillation seminorm controls from above the supremum norm. Indeed, for
every $n_0\in\NN_0$ we have
\[
	\sup_{n \in \NN_0}
	\abs{a_{n}}
	\leq
	\abs{a_{n_0}}+
	2\mathcal O \big(\seq{a_{n}}{n \in \NN_0}\big).
\]

In the two-parameter setting we would like to exploit the same kind of concepts. We shall study a variant
of two-parameter oscillation seminorm inspired by the one introduced in \cite{jrw}. Namely, given a lacunary
sequence $\big(N_k : k\in\NN\big) \subseteq \NN$ and a sequence $\big(\seq{a_{{n_1, n_2}}}{n_1, n_2 \in \NN_0}\big)$
an oscillation seminorm $\calO$ is defined by
\[
	\calO \big(\seq{a_{n_1, n_2}}{n_1, n_2 \in \NN_0}\big) =
	\bigg(
	\sum_{k\in\NN}
	\sup_{N_k \leq n_1, n_2 \leq N_{k+1}}
	\abs{a_{n_1, n_2} - a_{N_k,N_k}}^2
	\bigg)^{1/2}.
\]
Let us observe that the oscillation seminorm does not control the two-parameter supremum anymore.
A good counterexample illustrating this is a sequence
$\big(\seq{a_{n_1, n_2}}{n_1, n_2 \in \NN_0}\big)$ defined by
\begin{align*}
	a_{n_1, n_2} =
	\begin{cases}
		n_2	&	\text{if}\ n_1 = 0,\\
		0	&	\text{otherwise.}
	\end{cases}
\end{align*}
Indeed, $\calO\big(a_{n_1, n_2}: n_1, n_2 \in \NN_0)$ is finite for every lacunary sequence
$\big(\seq{N_k}{k \in \NN}\big)$ whereas
\[
	\sup_{n_1, n_2 \in \NN_0} \abs{a_{n_1, n_2}}=\infty.
\]
This shows a major difference between one- and multi-parameter settings. However, the oscillation
seminorm still remains useful in pointwise convergence questions (see Section \ref{sec:2} for details).

The second main result is the following.
\begin{main_theorem}
	\label{thm:2}
	For any lacunary sequence $\calN=\big(\seq{N_k}{k \in \NN}\big)$ there is a constant $C>0$ such that for any 
	finite set $\Lambda \subset Q_1^{-1}\ZZ \times Q_2^{-1}\ZZ$ with $Q_1,Q_2 \in \NN$ satisfying \eqref{eq:1},
	and all $f \in L^2\big(\RR^2\big)$ we have
	\[
		\Big\lVert
		\calO\Big(\seq{\sum_{\lambda \in \Lambda}
		\calF^{-1}
		\big(\ind{R_{n_1, n_2}^\lambda} \calF f \big)}{n_1, n_2 \in \NN_0}
		\Big)
		\Big\rVert_{L^2}
		\leq
		C
		\big(\log \log \big(Q_1 \sqrt{\abs{\Lambda}}\big)
		\cdot
		\log \log \big(Q_2 \sqrt{\abs{\Lambda}}\big)\big)^2
		\vnorm{f}_{L^2},
	\]
	where $R_{n_1, n_2}^\lambda = \lambda + A_{n_1} \times A_{n_2}$ and
	$\mathcal F$ is the Fourier transform operator on $\RR^2$.
\end{main_theorem}
The proofs of Theorem \ref{thm:1} and Theorem \ref{thm:2} consist of three steps. We shall
analyze both: the supremum and the oscillation seminorm in four regions; where the parameters
are small, large and finally mixed. In the last case, by the symmetry it suffices to
consider the region where the first parameter is small and the second is large.

In the one parameter case the regime with small parameters was estimated with the aid
of Rademacher--Menshov theorem, which asserts that there is an absolute constant $C>0$
such that for a given family of sets $U_1 \subseteq U_2 \subseteq\ldots\subseteq \RR$,
for any $N\in\NN$ and for any $f\in L^2\big(\RR\big)$ we have
\begin{align}
	\label{eq:25}
    \Big\|\sup_{1\le n \le N}
	\Big|\mathcal F^{-1}
	\Big(\ind{U_n}\mathcal Ff\Big)
	\Big|
	\Big\|_{L^2}\le C (\log N) \|f\|_{L^2}.
  \end{align}
In our situation we are going to exploit the same sort of ideas, however
the Rademacher--Menshov theorem must be adjusted to the two-parameter settings.
Due to independent parameters  we cannot hope for an analogous formulation with a nested
family of sets. Fortunately, it is still possible to prove a reasonable two-parameter
counterpart of \eqref{eq:25}. In the supremum case, Theorem \ref{th:1} implies that
there is a  constant $C>0$ such that for all $N_1, N_2 \in \NN$ and for all
$f\in L^2\big(\RR^2\big)$ we have
\begin{align*}
	\Big\|
	\sup_{1 \leq n_1 \leq N_1} \sup_{1 \leq n_2 \leq N_2}
	\Big|
	\calF^{-1}
	\Big(\ind{\bigcup_{\lambda \in \Lambda}R_{n_1, n_2}^{\lambda} } \calF f\Big)
	\Big|
	\Big\|_{L^2}
	\le
	C\big(\log N_1\big)\big(\log N_2\big)\|f\|_{L^2}.
  \end{align*}
The proof of Theorem \ref{th:1} is a consequence of a numerical inequality
(see Lemma \ref{lem:1}) which is of independent interests. To the authors' best
knowledge it is the first multi-parameter version of the Rademacher--Menshov theorem.

In the second case we analyze the region with large parameters and get a uniform bound
with respect to the size of the family $\Lambda$. Here it is important that
the frequencies are rational and the periodicity of the Fourier characters will be used
(see Theorem \ref{th:3}). This allows us to reduce the matters to suitable estimates
corresponding to Fej\'{e}r kernels (see Theorem \ref{th:2} and Theorem \ref{th:4}).
Then, in the oscillation case, the Fefferman--Stein vector-valued inequality for the
Hardy--Littlewood maximal operator transfers our problem to one-parameter
oscillation estimates of Fej\'{e}r kernels which complete the job.

Finally, in the mixed regime we need to combine the ideas from the first and
second case. Here we obtain the desired estimate with a logarithmic loss too.

An interesting question remains open whether there is a chance to relax the assumption
concerning rational frequencies in Theorem \ref{thm:1} and Theorem \ref{thm:2}.
 We have some partial results, in this directions, dealing with the case when
$f \in L^2\big(\RR^2\big)$ is assumed to be a tensor function.

\subsection{Notation}
Throughout the whole article, unless otherwise stated, we will write $A \lesssim B$
($A \gtrsim B$) if there is an absolute constant $C>0$ such that $A\le CB$ ($A\ge CB$).
Moreover, $C > 0$ will stand for a large positive constant whose value may vary from occurrence to
occurrence.

\section{Oscillation seminorm}
\label{sec:2}
Let us fix a lacunary sequence $\calN=\big(N_k : k \in \NN\big)$, i.e. a sequence satisfying
$\tau N_k \leq N_{k+1}$ for all $k \in \NN$ and some $\tau > 1$. An oscillation seminorm for a
two-parameter sequence $\big(\seq{a_{n_1, n_2}}{n_1, n_2 \in \NN_0}\big)$ is defined by
\[
	\calO\big(\seq{a_{n_1, n_2}}{n_1,n_2 \in \NN_0}\big)
	=
	\bigg(
	\sum_{k \in \NN} \sup_{N_k \leq n_1, n_2 \leq N_{k+1}}
	\big\lvert a_{n_1, n_2} - a_{N_k, N_k}\big\rvert^2
	\bigg)^{1/2}.
\]
We say that a two-parameter sequence of complex numbers
$\big(\seq{a_{n_1, n_2}}{n_1, n_2 \in \NN_0}\big)$ \emph{converges} to $a\in\CC$ if for every
$\varepsilon > 0$ there is $N \in \NN$ such that for all $n_1, n_2 \geq N$
\[
	\abs{a_{n_1, n_2} - a} \leq \varepsilon.
\]
In this case we write
\[
	\lim_{n_1, n_2 \to \infty}
	a_{n_1, n_2} = a.
\]
Thus $n_1, n_2 \to \infty$ is understood in the sense that $\min\{n_1, n_2\} \to \infty$.
We say that a two-parameter sequence $\big(a_{n_1, n_2}: n_1, n_2 \in \NN_0\big)$ is a \emph{Cauchy sequence}
if for all $\varepsilon > 0$ there is $N \in \NN$ such that for all $n_1, n_2, m_1, m_2 \geq N$
\[
	\abs{a_{n_1, n_2}-a_{m_1, m_2}} < \varepsilon.
\]
It is not difficult to see that every Cauchy sequence has a limit. But, unlike the one-parameter
situation, it is not true that every convergent two-parameter sequence must be \emph{bounded}.
To see this it suffices to consider
\begin{align*}
	a_{n_1, n_2} =
		\begin{cases}
			n_2 & \text{if}\ n_1 = 0,\\
			0   & \text{otherwise.}
		\end{cases}
\end{align*}
This is the main obstacle which prevents the two-parameter supremum norm from
being controlled from above by the oscillation seminorm. However, the two-parameter oscillation seminorm is still
a very useful object in problems involving pointwise convergence.
\begin{proposition}
	\label{prop:1}
	Suppose that $\big(a_{n_1, n_2}: n_1, n_2 \in \NN_0\big)$ is a sequence of complex numbers
	such that for every lacunary sequence $(N_k: k\in\NN)$ the oscillation seminorm
	$\calO\big(\seq{a_{n_1, n_2}}{n_1, n_2 \in\NN_0}\big)$ is finite. Then the limit
	$\lim_{n_1, n_2 \to \infty} a_{n_1, n_2}$ exists.
\end{proposition}
\begin{proof}
	Suppose that the limit does not exist. Since the sequence is not a Cauchy sequence, there
	exists $\varepsilon>0$ such that for every $N \in \NN$ there are $n_1, n_2, m_1, m_2 \ge N$
	satisfying
    \begin{align*}
		\abs{a_{n_1, n_2}-a_{m_1, m_2}} \ge 2\varepsilon.
    \end{align*}
	Hence, for some $n_1, n_2 \geq N$,
	\begin{equation}
		\label{eq:52}
		\abs{a_{n_1, n_2} - a_{N, N}} \geq \varepsilon.
	\end{equation}
	We are going to construct a sequence $(N_k : k \in \NN)$ such that for all $k \in \NN$
	\[
		\sup_{N_k \leq n_1, n_2 \leq N_{k+1}} \abs{a_{n_1, n_2} - a_{N_k, N_k}} \geq \varepsilon.
	\]
	Let $N_1 = 1$. Having chosen $N_k$, by \eqref{eq:52}, we can find $u_1^k, u_2^k \geq N_k$ so that
	\[
		\abs{a_{u_1^k, u_2^k} - a_{N_k, N_k}} \geq \varepsilon.
	\]
	Setting $N_{k+1} = 2 \max\{u_1^k, u_2^k\}$ we obtain
	\[
		\sup_{N_k \leq n_1, n_2 \leq N_{k+1}} \abs{a_{n_1, n_2} - a_{N_k, N_k}} \geq \varepsilon.
	\]
	Now, for every $K\in\NN$ we may estimate
	\[
		\varepsilon K^{1/2}
		\le
		\bigg(\sum_{k=1}^K
		\sup_{N_k \le n_1, n_2 \leq N_{k+1}}
		\abs{a_{n_1, n_2} - a_{N_k, N_k}}^2
		\bigg)^{1/2}
		\le
		\calO\big(\seq{a_{n_1, n_2}}{n_1, n_2 \in \NN_0}\big).
	\]
	This leads to a contradiction since $K$ may be taken as large as we wish.
\end{proof}
Now, for $w \in \NN^2$ we define
\begin{equation}
	\label{eq:11}
	\begin{aligned}
	U_{w}^{00} &= \{0, 1,\ldots, w_1-1\}\times \{0, 1,\ldots, w_2-1\},\\
	U_{w}^{01} &= \{0, 1,\ldots, w_1-1\}\times \{w_2, \ldots \},\\
	U_{w}^{10} &= \{w_1, \ldots \}\times \{0, 1,\ldots, w_2-1\},\\
	U_{w}^{11} &= \{w_1, \ldots \}\times \{w_2, \ldots \}.
	\end{aligned}
\end{equation}
The following lemma allows us to split the oscillation seminorm into four different regimes.
\begin{lemma}
	\label{lem:3}
	For a fixed $w \in \NN^2$ and each sequence of complex numbers
	$\big(a_{n_1, n_2}: n_1, n_2 \in \NN_0\big)$ we have
	\begin{align}
		\label{eq:42}
		\calO\big(a_{n_1, n_2}: n_1, n_2 \in \NN_0\big)
		\leq
		4 \sup_{n_1, n_2 \in \NN_0^2}|a_{n_1, n_2}|
		+\sum_{\mu \in \{0, 1\}^2}
		\calO_\mu\big(a_{n_1, n_2} : n_1, n_2 \in \NN_0\big),
    \end{align}
	where $K_\mu = \big\{k \in \NN: (N_k, N_k), (N_{k+1}, N_{k+1}) \in U^\mu_w \big\}$ and
	\[
		\calO_\mu \big(a_{n_1, n_2}: n_1, n_2 \in \NN_0\big)
		=
		\bigg(\sum_{k \in K_\mu}
		\sup_{N_k \leq n_1, n_2 \leq N_{k+1}}
		| a_{n_1, n_2}-a_{N_k, N_k}|^2
		\bigg)^{1/2}.
	\]
\end{lemma}
\begin{proof}
	First, let us observe that the diagonal $\{(n, n) : n \in \NN_0\}$ intersects $U_w^{01}$
	or $U_w^{10}$, but not both. Therefore, without loss of generality, we may assume that
	$K_{00}, K_{01}$ and $K_{11}$ are nonempty. Only four cases may occur:
	\begin{enumerate}
		\item there are $u, v\in \NN$ such that $K_{00}\cup \{u\}\cup
  K_{01}\cup\{v\}\cup K_{11}=\NN,$ and
			$(N_u, N_u) \in U_w^{00}$, $(N_{u+1}, N_{u+1}) \in U_w^{01}$, and
			$(N_v, N_v) \in U_w^{01}$, $(N_{v+1}, N_{v+1}) \in U_w^{11}$;
		\item there is $u \in \NN$ such that $K_{00}\cup \{u\}\cup
  K_{01}\cup K_{11}=\NN,$ and $(N_u, N_u) \in U_w^{00}$, and
			$(N_{u+1}, N_{u+1}) \in U_w^{01}$;
		\item there is $v \in \NN$ such that $K_{00}\cup
                  K_{01}\cup \{v\}\cup K_{11}=\NN,$
                  and $(N_v, N_v) \in U_w^{01}$,
			and $(N_{v+1}, N_{v+1}) \in U_w^{11}$;
		\item $K_{00}\cup K_{01}\cup K_{11}=\NN$.
	\end{enumerate}
	In the first case one can see that
	\begin{multline*}
 		\calO \big(a_{n_1, n_2}: n_1, n_2 \in\NN_0\big)
		\le \sum_{\mu \in \{0, 1\}^2}
		\bigg(\sum_{k \in K_\mu}
        \sup_{N_k \leq n_1, n_2 \leq N_{k+1}}
		| a_{n_1, n_2}-a_{N_k, N_k}|^2
		\bigg)^{1/2}\\
		+2\sup_{N_u \leq n_1, n_2 \leq N_{u+1}} |a_{n_1, n_2}|
		+2\sup_{N_v \leq n_1, n_2 \leq N_{v+1}} |a_{n_1, n_2}|,
	\end{multline*}
	which is dominated by the right-hand side of \eqref{eq:42}. For the remaining three cases we proceed
	analogously.
\end{proof}

\section{Two-parameter Rademacher--Menshov theorem}
In this section we prove a two-parameter version of the Rademacher--Menshov theorem. We start by proving
a numerical inequality which is interesting in its own right.

For a given sequence $a = \big(a_{n_1,n_2} : n_1,n_2 \in \NN_0\big)$ we define
the difference operators
\begin{gather*}
	\Delta^1_{n_1,n_2} (a) = a_{n_1,n_2} - a_{n_1-1, n_2}, \qquad
	\Delta^2_{n_1,n_2} (a) = a_{n_1,n_2} - a_{n_1,n_2-1},
\end{gather*}
and the double difference operator
\[
	\Delta_{n_1,n_2} (a) = a_{n_1,n_2} - a_{n_1,n_2-1}-a_{n_1-1,n_2} + a_{n_1-1,n_2-1}.
\]
Let us observe that $\Delta_{n_1,n_2} (a)
=\Delta^1_{n_1,n_2} \big(\Delta^2_{n_1,n_2} (a)\big)
=\Delta^2_{n_1,n_2} \big(\Delta^1_{n_1,n_2} (a)\big)$.
Moreover, for any two dyadic intervals $I_{j_1}^{i_1} = \big((j_1-1) 2^{i_1}, j_12^{i_1}\big]$,
$I_{j_2}^{i_2} = \big((j_2-1) 2^{i_2}, j_22^{i_2}\big]$ we have
$$
\sum_{k_1 \in I_{j_1}^{i_1}} \sum_{k_2 \in I_{j_2}^{i_2}} \Delta_{k_1,k_2}(a)
=
a_{j_12^{i_1}, j_22^{i_2}} - a_{(j_1-1)2^{i_1}, j_22^{i_2}} - a_{j_12^{i_1}, (j_2-1)2^{i_2}} + a_{(j_1-1)2^{i_1}, (j_2-1)2^{i_2}}.
$$
\begin{lemma}
	\label{lem:1}
	For every sequence of complex numbers $a=\big(a_{n_1,n_2}:  n_1, n_2 \in \NN_0 \big)$
	and all $s_1, s_2 \in \NN_0$
	\begin{multline}
		\label{eq:16}
		\sup_{0\le n_1< 2^{s_1}}\sup_{0\le n_2< 2^{s_2}}
		\abs{a_{n_1, n_2}}
		\le
		8\sum_{i_1 = 0}^{s_1} \sum_{i_2 = 0}^{s_2}
		\Big(
		\sum_{j_1 = 1}^{2^{s_1-i_1}}
		\sum_{j_2 = 1}^{2^{s_2-i_2}}
		\big\lvert
		\sum_{k_1 \in I_{j_1}^{i_1}} \sum_{k_2 \in I_{j_2}^{i_2}}
		\Delta_{k_1,k_2}(a)
		\big\rvert^2
		\Big)^{1/2}\\
		+
		2\sqrt{2}
		\sum_{i_1 = 0}^{s_1}
		\Big(
		\sum_{j_1 = 1}^{2^{s_1-i_1}}
		\big\lvert
		\sum_{k_1 \in I_{j_1}^{i_1}}
		\Delta^1_{k_1, n_2^0} (a)
		\big\rvert^2
		\Big)^{1/2}
		+
		2\sqrt{2}
		\sum_{i_2 = 0}^{s_2}
		\Big(
		\sum_{j_2 = 1}^{2^{s_2-i_2}}
		\big\lvert
		\sum_{k_2 \in I_{j_2}^{i_2}}
		\Delta^2_{n_1^0, k_2} (a)
		\big\rvert^2
		\Big)^{1/2}
		+\abs{a_{n_1^0, n_2^0}}
	\end{multline}
	for any $n_1^0 \in \{0, \ldots, 2^{s_1}-1\}$, $n_2^0 \in \{0, \ldots, 2^{s_2}-1\}$.
\end{lemma}
\begin{proof}
	First of all, we prove that for every sequence of complex numbers
	$b=\big(b_{n} : n \in \NN_0\big)$ and every $s\in\NN_0$ we have
	\begin{align}
    	\label{eq:6}
		\sup_{0 \leq n < 2^s} \sabs{b_n-b_{n_0}}
		\leq
		2\sup_{0 \leq m< n < 2^s} \sabs{b_n-b_m}\le
		2\sqrt{2}
		\sum_{i = 0}^s
		\bigg(
		\sum_{j = 1}^{2^{s-i}} \sabs{b_{j2^i} - b_{(j-1)2^i}}^2
		\bigg)^{1/2}
	\end{align}
	for any $n_0 \in \{0, \ldots, 2^s-1\}$. For this purpose we need
	the following combinatorial property: any interval $[m, n)$, with $0\le m<n< 2^s$ where
	$m, n \in \NN_0$, can be written as a disjoint union of dyadic subintervals, i.e. belonging
	to some
	\[
	\calI_i = \big\{[j 2^i, (j+1) 2^i) : 0 \leq j \leq 2^{s-i} - 1\big\}
	\]
	for $0 \leq i \leq s$, such that each length appears at most twice. For the proof, we start by
	setting $m_0 = m$. Having chosen $m_p$ we select $m_{p+1}$ in such a way that $[m_p, m_{p+1})$
	is the longest dyadic interval starting at $m_p$ and contained in $[m_p, n)$.
	We claim that, if $m_{p+1} - m_p \geq m_{p+2} - m_{p+1}$ then
	$m_{p+2} - m_{p+1} > m_{p+3} - m_{p+2}$. Suppose that, on the contrary,
	$m_{p+2} - m_{p+1} \leq m_{p+3} - m_{p+2}$. Then
	$$
	[m_{p+1}, 2m_{p+2} - m_{p+1}) \subset [m_{p+1}, n)
	$$
	and since $2 (m_{p+2} - m_{p+1})$ divides $m_{p+1}$ it contradicts with the choice of $m_{p+2}$.

    Now, we turn to the proof of \eqref{eq:6}. If $m<n$ we may write
	$$
	[m, n) = \bigcup_{p = 0}^{P} [u_p, u_{p+1})
	$$
	for some $P\ge1$, where each interval $[u_p, u_{p+1})$ is dyadic. Then
	$$
	\abs{b_{n} - b_{m}}
	\leq
	\sum_{p = 0}^{P}
	\abs{b_{u_{p+1}} - b_{u_p}}
	=
	\sum_{i = 0}^s
	\sum_{p : [u_p, u_{p+1}) \in \calI_i}
	\abs{b_{u_{p+1}} - b_{u_p}}.
	$$
	We note that the inner sum contains at most two terms, thus
	$$
	\big\lvert
	b_{n} - b_m
	\big\rvert
	\leq
	\sqrt{2}
	\sum_{i = 0}^s
	\bigg(\sum_{p : [u_p, u_{p+1}) \in \calI_i}
	\big\lvert
	b_{u_{p+1}} - b_{u_p}
	\big\rvert^2
	\bigg)^{1/2}
	$$
	which is bounded by the right-hand side of \eqref{eq:6} and the proof is completed.

	Next, we show how the inequality \eqref{eq:16} can be deduced from \eqref{eq:6}.
	Let $n_1^0 \in \{0, \ldots, 2^{s_1}-1\}$ and  $n_2^0 \in \{0, \ldots, 2^{s_2}-1\}$
	and observe that
    \begin{equation}
		\label{eq:7}
		\sup_{0\le n_1< 2^{s_1}}
		\sup_{0\le n_2< 2^{s_2}}
		\big|a_{n_1, n_2}\big|
		\le
		\sup_{0\le n_1< 2^{s_1}}\sup_{0\le n_2< 2^{s_2}}
		\big|a_{n_1, n_2}-a_{n_1,n_2^0}\big|+
		\sup_{0\le n_1< 2^{s_1}}
		\big|a_{n_1,n_2^0}-a_{n_1^0,n_2^0}\big|
		+\big|a_{n_1^0,n_2^0}\big|.
  	\end{equation}
	In view of \eqref{eq:6}, the second term of \eqref{eq:7} can be immediately dominated
	by the second sum on the right-hand side of
        \eqref{eq:16}. Applying now inequality
	\eqref{eq:6} to the inner supremum in the first term of \eqref{eq:7} one gets
	\begin{multline}
		\label{eq:21}
		\sup_{0\le n_1< 2^{s_1}}\sup_{0\le n_2< 2^{s_2}}
		\big|a_{n_1, n_2}-a_{n_1,n_2^0}\big|
		\le2\sqrt{2}
		\sup_{0\le n_1< 2^{s_1}}
		\sum_{i_2 = 0}^{s_2}
		\Big(
		\sum_{j_2 = 1}^{2^{s_2-i_2}}
		\big\lvert
		a_{n_1, j_2 2^{i_2}}-a_{n_1, (j_2-1)2^{i_2}}
		\big\rvert^2
		\Big)^{1/2}\\
		\le
		2\sqrt{2}
  		\sum_{i_2 = 0}^{s_2}
		\Big(
		\sum_{j_2 = 1}^{2^{s_2-i_2}}
		\sup_{0\le n_1< 2^{s_1}}
		\big\lvert
		a_{n_1, j_2 2^{i_2}}-a_{n_1, (j_2-1)2^{i_2}}
		-a_{n_1^0, j_22^{i_2}}+a_{n_1^0, (j_2-1)2^{i_2}}
		\big\rvert^2
		\Big)^{1/2}\\
		+2\sqrt{2}
		\sum_{i_2 = 0}^{s_2}
		\Big(
		\sum_{j_2 = 1}^{2^{s_2-i_2}}
		\big\lvert
		\sum_{k_2 \in I_{j_2}^{i_2}}
		\Delta^2_{n_1^0, k_2} (a)
		\big\rvert^2
		\Big)^{1/2}.
	\end{multline}
	Next, using \eqref{eq:6} for the second time one can dominate \eqref{eq:21} by the first sum
	from \eqref{eq:16}. Indeed,  taking $b_{n_1}=a_{n_1, j_22^{i_2}} -a_{n_1, (j_2-1)2^{i_2}}$
	inequality \eqref{eq:6} yields
	\begin{multline*}
  		\sum_{i_2 = 0}^{s_2}
		\Big(
		\sum_{j_2 = 1}^{2^{s_2-i_2}}
		\sup_{0\le n_1< 2^{s_1}}
		\big\lvert
		a_{n_1, j_22^{i_2}}-a_{n_1, (j_2-1)2^{i_2}}
		-a_{n_1^0, j_22^{i_2}}+a_{n_1^0, (j_2-1)2^{i_2}}
		\big\rvert^2
		\Big)^{1/2}\\
		\le2\sqrt{2}
		\sum_{i_2 = 0}^{s_2}
		\bigg(
		\sum_{j_2 = 1}^{2^{s_2-i_2}}
		\bigg(\sum_{i_1 = 0}^{s_1}
		\Big(
		\sum_{j_1 = 1}^{2^{s_1-i_1}}
		\big\lvert
		\sum_{k_1 \in I_{j_1}^{i_1}} \sum_{k_2 \in I_{j_2}^{i_2}}
		\Delta_{k_1,k_2}(a)
		\big\rvert^2
		\Big)^{1/2}\bigg)^2\bigg)^{1/2}\\
		\le2\sqrt{2}
		\sum_{i_1 = 0}^{s_1}
		\sum_{i_2 = 0}^{s_2}
		\bigg(
		\sum_{j_1 = 1}^{2^{s_1-i_1}}
		\sum_{j_2 = 1}^{2^{s_2-i_2}}
		\big\lvert
		\sum_{k_1 \in I_{j_1}^{i_1}}
		\sum_{k_2 \in I_{j_2}^{i_2}}
		\Delta_{k_1,k_2}(a)
		\big\rvert^2
		\bigg)^{1/2}.
	\end{multline*}
	The last estimate follows from Minkowski's inequality. This finishes the proof
	of Lemma \ref{lem:1}.
\end{proof}

For a function $f \in L^1\big(\RR^d\big)$ let
$$
\calF{f}(\xi) = \int_{\RR^d} e^{-2\pi i \sprod{\xi}{x}} f(x) {\: \rm d}x.
$$
By $\mathcal F^{-1}$ we denote the inverse Fourier transform. In our
setup $d=1$ or $d=2$, but it will be always clear from the context. The following
theorem gives a variant of two-parameter Rademacher--Menshov inequality.
\begin{theorem}
	\label{th:1}
	Let $m = \big(m_{n_1,n_2} : n_1, n_2 \in \NN_0\big)$ be a sequence of measurable functions
	on $\RR^2$ satisfying
	\begin{align}
		\label{eq:6.1}
		&
		\sup_{\xi \in \RR^2}
		\sum_{n_1,n_2 \in \NN}
		\big\lvert \Delta_{n_1,n_2}(m) (\xi)
		\big\rvert
		\leq B, \qquad\\
		\label{eq:6.2}
		&
		\sup_{\xi \in \RR^2}
		\sum_{n_1 \in \NN}
		\big\lvert \Delta^1_{n_1,0}(m)(\xi) \big\rvert
		\leq B, \qquad
		\sup_{\xi \in \RR^2}
		\sum_{n_2 \in \NN}
		\big\lvert \Delta^2_{0, n_2}(m)(\xi) \big\rvert
		\leq B,\\
		\label{eq:6.4}
		&
		\sup_{\xi \in \RR^2} \lvert m_{0,0}(\xi) \rvert \leq B.
	\end{align}
	Then there is a constant $C>0$ such that for all $N_1, N_2 \in\NN$ and all
	$f \in L^2\big(\RR^2\big)$
	$$
	\Big\lVert
	\sup_{0\le n_1< N_1}
	\sup_{0\le n_2< N_2}
	\big|\calF^{-1}\big(m_{n_1,n_2} \calF f\big)\big|
	\Big\rVert_{L^2}
	\leq
	C B (\log N_1)(\log N_2) \vnorm{f}_{L^2}.
	$$
\end{theorem}
\begin{proof}
	We may assume that $N_1 = 2^{s_1}$, $N_2 = 2^{s_2}$, for some $s_1, s_2 \in \NN$. Then,
	by Lemma \ref{lem:1} we can estimate
	\begin{multline}
		\label{eq:5}
		\Big\lVert
		\sup_{0\le n_1< N_1}\sup_{0\le n_2< N_2}
		\big|\calF^{-1}\big(m_{n_1,n_2} \calF f\big)\big|
		\Big\rVert_{L^2}\\
		\lesssim
		\sum_{i_1 = 0}^{s_1} \sum_{i_2 = 0}^{s_2}
		\Big(
		\sum_{j_1 = 1}^{2^{s_1-i_1}} \sum_{j_2 = 1}^{2^{s_2-i_2}}
		\Big\lVert
		\sum_{p_1 \in I_{j_1}^{i_1}} \sum_{p_2 \in I_{j_2}^{i_2}}
		\calF^{-1}\big(\Delta_{p_1,p_2}(m) \calF f\big)
		\Big\rVert_{L^2}^2
		\Big)^{1/2}\\
		+
		\sum_{i_1 = 0}^{s_1}
		\Big(
		\sum_{j_1 = 1}^{2^{s_1-i_1}}
		\Big\lVert
		\sum_{p_1 \in I_{j_1}^{i_1}} \calF^{-1}\big(\Delta^1_{p_1, 0}(m) \calF f\big)
		\Big\rVert_{L^2}^2
		\Big)^{1/2}\\
		+
		\sum_{i_2 = 0}^{s_2}
		\Big(
		\sum_{j_2 = 1}^{2^{s_2-i_2}}
		\Big\lVert
		\sum_{p_2 \in I_{j_2}^{i_2}} \calF^{-1}\big(\Delta^2_{0,p_2}(m) \calF f\big)
		\Big\rVert_{L^2}^2
		\Big)^{1/2}
		+ \lVert \calF^{-1}\big(m_{0,0} \calF f \big)\rVert_{L^2}.
	\end{multline}
	Let us fix $i_1 \in \{0, \ldots, s_1\}$ and $i_2 \in \{0, \ldots, s_2\}$. By Plancherel's
	theorem and \eqref{eq:6.1}, for any $j_1 \in \big\{1, \ldots, 2^{s_1-i_1}\big\}$ and
	$j_2 \in \big\{1, \ldots, 2^{s_2-i_2}\big\}$ we have
	\begin{multline*}
		\Big\lVert
		\sum_{p_1 \in I_{j_1}^{i_1}} \sum_{p_2 \in I_{j_2}^{i_2}}
		\calF^{-1}\big(\Delta_{p_1,p_2} (m) \calF f\big)
		\Big\rVert_{L^2}^2
		\leq
		\sum_{p_1,q_1 \in I_{j_1}^{i_1}} \sum_{p_2, q_2 \in I_{j_2}^{i_2}}
		\int \sabs{\Delta_{p_1,p_2}(m)(\xi)} \cdot
		\sabs{\Delta_{q_1,q_2} (m) (\xi)} \cdot \sabs{\calF f(\xi)}^2 {\: \rm d}\xi\\
		\leq
		B
		\sum_{p_1 \in I_{j_1}^{i_1}} \sum_{p_2 \in I_{j_2}^{i_2}}
		\int \sabs{\Delta_{p_1,p_2}(m)(\xi)} \cdot
		\sabs{\calF f(\xi)}^2 {\: \rm d}\xi.
	\end{multline*}
	Hence,
        \begin{multline*}
        \sum_{j_1 = 1}^{2^{s_1-i_1}} \sum_{j_2=1}^{2^{s_2-i_2}}
		\Big\lVert
		\sum_{p_1 \in I_{j_1}^{i_1}} \sum_{p_2 \in I_{j_2}^{i_2}}
		\calF^{-1}\big(\Delta_{p_1,p_2} (m) \calF f\big)
		\Big\rVert_{L^2}^2\\
		\leq
		B
		\sum_{j_1 = 1}^{2^{s_1-i_1}} \sum_{j_2=1}^{2^{s_2-i_2}}
		\sum_{p_1 \in I_{j_1}^{i_1}} \sum_{p_2 \in I_{j_2}^{i_2}}
		\int \sabs{\Delta_{p_1,p_2}(m)(\xi)} \cdot
		\sabs{\calF f(\xi)}^2 {\: \rm d}\xi
    \end{multline*}
	what is bounded by $B^2 \vnorm{f}_{L^2}^2$. By summing up over $i_1$ and $i_2$ one can
	shows that the first term in \eqref{eq:5} is bounded by $B s_1 s_2 \vnorm{f}_{L^2}$.
	Similar arguments based on \eqref{eq:6.2} allows to bound the second and the third term
	in \eqref{eq:5}.
\end{proof}

\begin{remark}
	For the proof of Theorem \ref{th:1}, it is enough to assume that the square functions
	\begin{equation}
		\calS f
		=
		\Big(
		\sum_{n_1,n_2 \in \NN} \big\lvert \calF^{-1}\big(\Delta_{n_1,n_2}(m) \calF f\big)\big\rvert^2
		\Big)^{1/2},
	\end{equation}
	and
	\begin{equation}
		\calS^1 f
		=
		\Big(
		\sum_{n_1 \in \NN} \big\lvert \calF^{-1}\big(\Delta^1_{n_1, 0}(m) \calF f\big) \big\rvert^2
		\Big)^{1/2},
		\qquad
		\calS^2 f
		=
		\Big(
		\sum_{n_2 \in \NN} \big\lvert \calF^{-1}\big(\Delta^2_{0, n_2}(m) \calF f\big) \big\rvert^2
		\Big)^{1/2},
	\end{equation}
	are bounded on $L^2\big(\RR^2\big)$.
\end{remark}

\section{Some $L^p$ estimates}
This section is devoted to study some general estimates which may be interesting in their
own rights. The argument in the proof of Theorem \ref{th:2} is inspired by an observation
due to Zygmund \cite[p. 164]{zyg2} (see also Sj\"olin \cite[Theorem 7.3]{sjol}).
\begin{theorem}
	\label{th:2}
	For each $p \in (1, \infty)$ there is a constant $C_p > 0$ such that for all
	$f \in L^1\big(\RR^2\big) \cap L^p\big(\RR^2\big)$
	\begin{equation}
		\label{eq:10}
        \Big\lVert
        \sup_{n_1, n_2\in\ZZ}
		\big|\calF^{-1}\big(\ind{R_{n_1,n_2}} \calF f\big)\big|
        \Big\rVert_{L^p}
		\leq
		C_p
		\vnorm{f}_{L^p}
	\end{equation}
	where $R_{n_1, n_2} = A_{n_1} \times A_{n_2}$.
\end{theorem}
\begin{proof}
	Let us recall the operators defined by multipliers $m_0$ and $m_1$ where
    \begin{align*}
        m_0 = \sum_{n \in 2 \ZZ} \big(\ind{A_{n-1}} - \ind{A_n}\big), \qquad
		m_1 = \sum_{n \in 2 \ZZ + 1} \big(\ind{A_{n-1}} - \ind{A_n}\big)
    \end{align*}
	are bounded on $L^p(\RR)$ for all $p\in(1, \infty)$. Indeed, it follows from the
	Littlewood--Paley theory that if $\mathfrak m \in L^{\infty}(\RR)$ and
	$\mathfrak m_j(\xi)=\mathfrak m(\xi)\ind{I_j}(\xi)$ where $I_j=A_{j-1} \setminus A_j$ then there is a
	constant $C_p^1>0$ such that all $f\in L^p(\RR)$
	\[
		\big\|\calF^{-1}\big(\mathfrak m \calF f\big)\big\|_{L^p}\le C_p^1\|f\|_{L^p},
	\]
	if and only if for some $C_p^2>0$ and all sequences $\big(f_j: j\in\ZZ\big)$ of functions
	in $L^p(\RR)$
	\[
		\Big\|
		\Big(\sum_{j\in\ZZ}\big|\calF^{-1}\big(\mathfrak m_j\calF f_j\big)\big|^2
		\Big)^{1/2}
		\Big\|_{L^p}
		\le C_p^2\Big\|\Big(\sum_{j\in\ZZ}|f_j|^2\Big)^{1/2}\Big\|_{L^p}.
	\]
	Moreover, the operator corresponding to the multiplier $\ind{[a, b]}$ is equal to
	
\[ i/2 \big(M_a\mathcal H M_{-a}-M_b\mathcal H M_{-b}\big),\]
 where $\mathcal H$ denotes the
	Hilbert transform and $M_af(x)=e^{2\pi i a x}f(x)$. Therefore, by applying vector-valued
	inequality for multipliers $\mathfrak m=m_0$ or $\mathfrak
        m=m_1$ one obtains the desired $L^p$ bounds.

        For two functions $a$ and $b$ we define $(a \otimes b)(x,y) = a(x) b(y)$.
        Hence, for any $\delta=\big(\delta_1, \delta_2\big) \in \{0, 1\}^2$ the multiplier
	\[
		m_{\delta} = m_{\delta_1} \otimes m_{\delta_2}
	\]
	is bounded on $L^p\big(\RR^2\big)$  for all $p\in(1, \infty)$. Since $m_{00} + m_{01}  + m_{10} + m_{11} \equiv 1$
	it is enough to prove \eqref{eq:10} for $\calF^{-1}\big( m_{\delta} \calF f\big)$ instead
	of $f$. Without loss of generality, we may assume that $\delta = (0,1)$, i.e.
	$f \in L^1\big(\RR^2\big) \cap L^p\big(\RR^2\big)$ and
    \begin{align}
		\label{eq:13}
		\supp \calF f \subseteq \bigcup_{n_1 \in 2 \ZZ}\bigcup_{n_2 \in 2 \ZZ+1}
		\big(A_{n_1-1} \setminus A_{n_1} \big) \times \big(A_{n_2-1} \setminus A_{n_2}\big).
    \end{align}
	Next, for any $D > 0$ and $x \in \RR$ we have
    \begin{align}
        \label{eq:15}
        \int_0^D \ind{[-a, a]}(x) {\: \rm d}a =
		\begin{cases}
			D - \abs{x} & \text{ if } \abs{x} \leq D,\\
			0 & \text{otherwise.}
		\end{cases}
    \end{align}
	Hence, if $\abs{x} \leq D < D'$ we obtain
	$$
	\int_D^{D'} \ind{[-a, a]}(x) {\: \rm d} a =(D' - D)\ind{[-D, D]}(x).
	$$
        This implies, in view of \eqref{eq:13}, that for any $n_1, n_2 \in \ZZ$
	\begin{equation*}
		\ind{R_{n_1,n_2}}(\xi) \cdot \calF f(\xi) =
		\bigg(\frac{1}{D_{n_1}^1 D_{n_2}^2}
		\int_{D_{n_1}^1}^{D_{n_1}^1}
		\int_{D_{n_2}^2}^{D_{n_2}^2}
		\ind{[-a, a]}(\xi_1)
		\ind{[-b, b]}(\xi_2)
		{\: \rm d}a {\: \rm d}b\bigg)\cdot
		\calF f(\xi)
	\end{equation*}
	where $D_n^1 = 2^{-2 \lceil (n+1)/2 \rceil}$ and $D_n^2 =
        2^{-2 \lfloor (n+1)/2 \rfloor-1}$ for $n\in\ZZ$. Therefore,
    \[
		\calF^{-1} \big(\ind{R_{n_1,n_2}} \calF f \big)
		=
		\Big(\big(2 \sigma_{2D_{n_1}^1} - \sigma_{D_{n_1}^1}\big)
		\otimes
		\big(2 \sigma_{2D_{n_2}^2} - \sigma_{D_{n_2}^2}\big)\Big)
		* f
	\]
	where $\sigma_D$ is a continuous Fej\'{e}r kernel, defined for $x \in \RR$ by
	\[
		\sigma_D(x) = \frac{1}{D} \bigg(\frac{\sin(\pi D x)}{\pi x}\bigg)^2.
	\]
	Let $\mathcal Mf(x)=\sup_{n\in \ZZ} |M_{2^n}f(x)|$ be the Hardy--Littlewood maximal
	function where
	\begin{align*}
		M_{D}f(x) =
		\frac{1}{2D}
		\int_{-D}^{D}
		f(x - y) {\: \rm d} y
	\end{align*}
	for $x \in \RR$ and $D>0$. Since for every $g\in L^p(\RR)$ we have the following pointwise bound
	\[
		\sup_{n \in \ZZ} |\sigma_{2^n}*g(x)|
		\lesssim
		\mathcal M(|g|)(x),
	\]
	we immediately conclude that
	\[
		\sup_{n_1, n_2\in\ZZ}
		\big|
		\calF^{-1} \big(\ind{R_{n_1,n_2}} \calF f \big)(x_1, x_2)\big|
		\lesssim
		\calM\big(g_{x_1}\big) (x_2),
	\]
	where $g_{x_1}(y) = \calM \big(f(\: \cdot \:,
        y)\big)(x_1)$. The remaining three cases for $\delta\in\{0,
        1\}^2\setminus\{(0, 1)\}$ can be  proved analogously and this completes the proof of Theorem \ref{th:2}.
\end{proof}
Next, we establish a variant of Theorem \ref{th:2} with the oscillation seminorm rather than supremum.
In the one-parameter theory the supremum is controlled from above by the oscillation seminorm. In the
two-parameter setup we do not have this property anymore so we have to provide a second proof.
\begin{theorem}
	\label{th:4}
	Let $p \in (1, \infty)$ and $\big(N_k: k\in\NN\big)$ be a lacunary sequence. Then there is
	a constant $C_p > 0$ such that for all $f \in L^1\big(\RR^2\big) \cap L^p\big(\RR^2\big)$
	\begin{equation*}
		\big\lVert
        \calO\big(
		\calF^{-1}\big(\ind{R_{n_1,n_2}} \calF f
		\big): n_1, n_2\in\NN_0\big)
        \big\rVert_{L^p}
		\leq
		C_p
		\vnorm{f}_{L^p}.
	\end{equation*}
\end{theorem}
\begin{proof}
	We follow the notation from Theorem \ref{th:2}. Again, we may assume that
	$f \in L^1\big(\RR^2\big) \cap L^p\big(\RR^2\big)$ and satisfies \eqref{eq:13}, thus
	\[
        \calF^{-1} \big(\ind{R_{n_1,n_2}} \calF f \big)
		=
		\big(\big(2 \sigma_{2D_{n_1}^1} - \sigma_{D_{n_1}^1}\big)
		\otimes
		\big(2 \sigma_{2D_{n_2}^2} - \sigma_{D_{n_2}^2}\big)\big) * f.
	\]
        The proof will be completed if we show that for any $p\in(1,
        \infty)$ there is a constant $C_p>0$ such that
        \[
        \big\lVert\calO\big(\big(\sigma_{s_{n_1}} \otimes \sigma_{t_{n_2}}\big) * f: n_1, n_2 \in \NN_0\big)
		\big\rVert_{L^p}\le C_p \|f\|_{L^p}
        \]
        for any $s_{n_1}\simeq 2^{n_1}$ and  $t_{n_2}\simeq 2^{n_2}$,
        since $D_{n_1}^1\simeq 2^{n_1}$ and   $D_{n_2}^2\simeq 2^{n_2}$.
	We notice that
	\begin{multline}
		\label{eq:51}
		\big\lVert
		\calO\big(\big(\sigma_{s_{n_1}} \otimes \sigma_{t_{n_2}}\big) * f: n_1, n_2 \in \NN_0\big)
		\big\rVert_{L^p} \\
		\le
		\Big\|
		\Big(
		\sum_{k \in \NN}
		\sup_{N_k \leq n_1, n_2 \leq N_{k+1}}
		\big|\big(\sigma_{s_{n_1}} \otimes \big(\sigma_{t_{n_2}} - \sigma_{t_{N_k}}\big)\big) *f \big|^2
		\Big)^{1/2}
		\Big\|_{L^p}\\
		+
		\Big\|\Big(
		\sum_{k\in\NN}
		\sup_{N_k \leq n_1 \leq N_{k+1}}
		\big|\big(\big(\sigma_{s_{n_1}} - \sigma_{s_{N_k}}\big) \otimes \sigma_{t_{N_k}}\big) * f \big|^2
		\Big)^{1/2}
		\Big\|_{L^p}.
	\end{multline}
	Let us consider the first term in \eqref{eq:51}. We have
	\begin{align*}
		\sup_{N_k \leq n_1, n_2 \leq N_{k+1}}
		\big|\big(\sigma_{s_{n_1}} \otimes \big(\sigma_{t_{n_2}} - \sigma_{t_{N_k}}\big)\big) * f(x_1,x_2) \big|
		\leq
		\calM\Big(
		\sup_{N_k \leq n_2 \leq N_{k+1}}
		\big|
		\big(\sigma_{t_{n_2}} - \sigma_{t_{N_k}}\big) *_2 f(\:\cdot\:, x_2)
		\big|
		\Big)(x_1),
	\end{align*}
	where $*_2$  denotes the convolution taken with respect to the second variable. Therefore,
	by the Fefferman--Stein vector-valued inequality (see \cite{fs}, see also \cite[Theorem 4.6.6]{graf1})
	we obtain
	\begin{multline*}
		\Big\|
		\Big(
		\sum_{k \in \NN}
		\sup_{N_k \leq n_1, n_2 \leq N_{k+1}}
		\big|\big(\sigma_{s_{n_1}} \otimes \big(\sigma_{t_{n_2}} - \sigma_{t_{N_k}}\big)\big) *f \big|^2
		\Big)^{1/2}
		\Big\|_{L^p}\\
		\leq
		\Big\|
		\Big(
		\sum_{k \in \NN}
		\calM\Big(
		\sup_{N_k \leq n_2 \leq N_{k+1}}
		\big|
		\big(\sigma_{t_{n_2}} - \sigma_{t_{N_k}}\big) *_2 f(\:\cdot\:, x_2)
		\big|
		\Big)(x_1)^2
		\Big)^{1/2}
		\Big\|_{L^p({\rm d}x_1 {\rm d}x_2)} \\
		\lesssim
		\big\|
		\calO\big(\sigma_{t_{n_2}} *_2 f : n_2 \in \NN_0 \big)
        \big\|_{L^p}
	\end{multline*}
	We may apply analogous argument to the second term in
        \eqref{eq:51}. Therefore, it suffices to show that there is
        $C_p > 0$ such that for all $g\in L^p(\RR)$
	\[
	    \big\|\calO\big(\sigma_{s_n} * g : n \in \NN_0\big)\big\|_{L^p}
		\leq
		C_p
		\vnorm{g}_{L^p}.
	\]
	Let us observe that
	\[
		\big\|\calO\big(\sigma_{s_n} * g : n \in \NN_0\big) \big\|_{L^p}
		\leq
		\big\|\calO\big(M_{2^n} g : n \in \NN_0\big)\big\|_{L^p}
		+
		\Big\|
		\Big(
		\sum_{n \in \NN_0} \big|\sigma_{s_n} * g - M_{2^n}g \big|^2
		\Big)^{1/2}
		\Big\|_{L^p}.
	\]
	According to \cite{jkrw}, we have
	\[
		\big\|\calO\big(M_{2^n} g : n \in \NN_0\big) \big\|_{L^p}\
		\lesssim
		\vnorm{g}_{L^p}.
	\]
	Hence, it suffices to show that there is a constant $C_p > 0$ such that for all $g \in L^p(\RR)$
	\[
		\Big\|\Big(\sum_{n \in \NN_0} \big|\big(\sigma_{s_n}  - K_{2^n}\big)*g \big|^2 \Big)^{1/2}
		\Big\|_{L^p}
		\leq C_p
		\vnorm{g}_{L^p}
	\]
	where $K_{D} * g = M_{D} g$ and $K_D(x)=(2D)^{-1}\ind{[-D, D]}(x)$. Let $S_j$ be a Littlewood--Paley projection
	$\calF(S_j g)(\xi)=\varphi_j(\xi)\calF g (\xi)$ associated with $\big(\varphi_j: j\in\ZZ\big)$
	a smooth partition of unity of $\RR\setminus\{0\}$ such that
        for each $j \in \ZZ$ we have
	$0\le \varphi_j\le 1$ and
	\[
		\supp{\varphi_j}
		\subseteq
		\big\{\xi\in\RR: 2^{-j-1} < \abs{\xi} < 2^{-j+1} \big\}.
	\]
	Then we have
	\begin{align*}
  		\Big\|
		\Big(\sum_{n \in \NN_0}
		\big|\big(\sigma_{s_n} - K_{2^n} \big) * g\big|^2\Big)^{1/2}
		\Big\|_{L^p}
		\leq
		\sum_{j \in \ZZ}
		\Big\|
		\Big(
		\sum_{n \in \NN_0}
		\big|
		\big(\sigma_{s_n} - K_{2^n}\big) * S_{j+n} g
		\big|^2
		\Big)^{1/2}
		\Big\|_{L^p}.
	\end{align*}
	We claim that there are $C_p > 0$ and $\delta_p > 0$ such that for every $j \in \ZZ$
	\begin{align}
		\label{eq:34}
		\Big\|
		\Big(\sum_{n \in \NN_0}
		\big|\big(\sigma_{s_n} - K_{2^n}\big) * S_{n+j} g \big|^2\Big)^{1/2}
		\Big\|_{L^p}
		\leq
		C_p
		2^{-\delta_p \abs{j}}
		\vnorm{g}_{L^p}.
	\end{align}
	Again, by the Fefferman--Stein vector-valued inequality and the boundedness of the square
	function associated with $(S_j : j \in \ZZ)$ we have
	\begin{multline}
		\label{eq:35}
		\Big\|
		\Big(
		\sum_{n \in \NN_0}
		\big|\big(\sigma_{s_n} - K_{2^n} \big) * S_{n+j} g \big|^2
		\Big)^{1/2}
		\Big\|_{L^p}
		\lesssim
		\Big\|
		\Big(
		\sum_{n \in \NN_0}
		\calM\big(S_{n+j} g\big)^2
		\Big)^{1/2}
		\Big\|_{L^p}\\
		\lesssim
		\Big\|
		\Big(
		\sum_{j \in \ZZ} \big| S_j g \big|^2
		\Big)^{1/2}
		\Big\|_{L^p}
		\lesssim
		\vnorm{g}_{L^p}.
	\end{multline}
	Next, for $p = 2$ we can refine the estimates \eqref{eq:35}. Indeed, by \eqref{eq:15}, we have
	\[
		\calF\sigma_{D}(\xi)=
		\begin{cases}
			1 - D^{-1} \abs{\xi} & \text{ if } \abs{\xi} \leq D,\\
			0 & \text{otherwise,}
		\end{cases}
	\]
	and
	\[
		\calF K_{D}(\xi)=
		\frac{\sin(2\pi i D \xi)}{2\pi D \xi}=1+O\big(D^2 |\xi|^2\big).
	\]
	Therefore, for some $\delta > 0$ we have
	\begin{align}
		\label{eq:38}
		\big|\calF\sigma_{s_n}(\xi)-\calF K_{2^n}(\xi)\big|
		\lesssim
		\min\big\{1, |2^n\xi|^{\delta}, |2^n\xi|^{-\delta}\big\}.
	\end{align}
	Applying Plancherel's theorem we obtain
	\begin{align}
		\label{eq:39}
		\Big\|\Big(\sum_{n \in \NN_0} \big| \big(\sigma_{s_n} - K_{2^n}\big)*S_{n+j} g \big|^2 \Big)^{1/2}
		\Big\|_{L^2}^2
		\lesssim
		2^{-\delta \abs{j}}
		\vnorm{g}_{L^2}^2,
	\end{align}
	since \eqref{eq:38} implies
	\begin{align*}
		\sum_{n\in\NN_0}
		\Big\|
		\big(
		\calF\sigma_{s_n}-\calF K_{2^n}
		\big)
		\varphi_{j+n}
		\calF g
		\Big\|_{L^2}^2
		\lesssim
		2^{-\delta \abs{j}}
		\vnorm{g}_{L^2}^2.
	\end{align*}
	Finally, by interpolation between \eqref{eq:39} and \eqref{eq:35} we obtain \eqref{eq:34}
	which finishes the proof.
\end{proof}

\section{Two-parameter logarithmic lemma}
Let $\Lambda$ be a finite subset of $Q_1^{-1}\ZZ \times Q_2^{-1}\ZZ$ for some
$Q_1,Q_2 \in \NN$. Suppose that for any  $\lambda, \lambda' \in \Lambda$,
if $\lambda \neq \lambda'$ then
\begin{equation}
	\label{eq:3}
	\norm{\lambda - \lambda'}
	= \max\{\abs{\lambda_1 - \lambda_1'}, \abs{\lambda_2 - \lambda_2'}\} \geq 1.
\end{equation}
Let $R_{n_1,n_2}^\lambda = \lambda + A_{n_1} \times A_{n_2}$. We are going to show the following.
\begin{theorem}
	\label{th:3}
	There is a constant $C>0$ such that for all $f \in L^2\big(\RR^2\big)$
	\[
		\Big\lVert\sup_{n_1, n_2 \in \NN_0}
		\Big|\sum_{\lambda \in \Lambda}
		\calF^{-1}\big(
		\ind{R_{n_1,n_2}^\lambda} \calF f \big)
		\Big\lvert
		\Big\rVert_{L^2}
		\leq
		C
		\log \log \big(Q_1 \sqrt{\abs{\Lambda}}\big)
		\cdot \log \log \big(Q_2 \sqrt{\abs{\Lambda}}\big)
		\vnorm{f}_{L^2}.
	\]
	Moreover, for every lacunary sequence $\calN = \big(N_k : k \in \NN\big)$ there is a constant
	$C > 0 $ such that for all $f \in L^2\big(\RR^2\big)$ we have
	\[
		\Big\lVert\calO\Big(\seq{\sum_{\lambda \in \Lambda}
		\calF^{-1}
		\big(\ind{R_{n_1,n_2}^\lambda} \calF f \big)}{n_1, n_2 \in \NN_0}\Big)
		\Big\rVert_{L^2}
		\leq
		C
		\big(\log \log \big(Q_1 \sqrt{\abs{\Lambda}}\big)
		\cdot
		\log \log \big(Q_2 \sqrt{\abs{\Lambda}}\big)\big)^2
		\vnorm{f}_{L^2}.
	\]
\end{theorem}
\begin{proof}
	Let us define $S_1 = 2^{s_1}$ and $S_2 = 2^{s_2}$ where
	$s_1 = \big\lceil \log_2 \log_2 Q_1 \sqrt{\abs{\Lambda}}\big\rceil$ and
	$s_2 = \big\lceil \log_2 \log_2 Q_2 \sqrt{\abs{\Lambda}}\big\rceil$.
	We set $w = (S_1, S_2)$ and split $\NN_0$ into four regions as in \eqref{eq:11}.
	By Lemma \ref{lem:3} we have
	\begin{multline*}
		\calO\Big(\sum_{\lambda \in \Lambda} \calF^{-1}
		\big(\ind{R_{n_1, n_2}^\lambda} \calF f \big) : n_1, n_2 \in \NN_0\Big)\\
		\lesssim
		\sum_{\mu \in \{0,1\}^2}
		\Big(
		\sup_{n_1, n_2 \in U_w^\mu}
		\Big|\sum_{\lambda \in \Lambda} \calF^{-1}\big(\ind{R_{n_1, n_2}} \calF f\big)\Big|
		+
		\calO_\mu\Big(\sum_{\lambda \in \Lambda} \calF^{-1}
		\big(\ind{R_{n_1, n_2}^\lambda} \calF f \big) : n_1, n_2 \in \NN_0\Big)
		\Big).
	\end{multline*}
	Since the regions $U_w^{01}$ and $U_w^{10}$ are mutually symmetric we shall only treat
	with the region $U_w^{01}$. The proof is divided into three parts according to which region, $U_w^{00}$, $U_w^{01}$ or $U_w^{11}$.

	\vspace*{1ex}
	\noindent
	{\bf The region $U_w^{00}$.} Since the elements of $\Lambda$ satisfies  separation condition
	\eqref{eq:3}, for all $n_1, n_2 \in \NN_0$, the rectangles $R_{n_1,n_2}^{\lambda}$ and
    $R_{n_1,n_2}^{\lambda'}$ are disjoint whenever $\lambda \neq \lambda'$. Moreover,
	the function
	\[
		\ind{R_{n_1-1,n_2-1}} - \ind{R_{n_1-1, n_2}} - \ind{R_{n_1,n_2-1}} + \ind{R_{n_1,n_2}}
	\]
	has a support inside $R_{n_1-1, n_2-1} \setminus \big(R_{n_1-1, n_2} \cup R_{n_1, n_2-1}\big)$.
	Thus the sequence $\big(\ind{R_{n_1,n_2}} : n_1,n_2 \in \NN_0\big)$ satisfies
	\eqref{eq:6.1}--\eqref{eq:6.4}, and by Theorem \ref{th:1}, we conclude
    \begin{align}
		\label{eq:43}
		\Big\lVert
		\sup_{(n_1, n_2) \in U_w^{00}}
		\Big|\sum_{\lambda \in \Lambda}
		\calF^{-1}
		\big(\ind{R_{n_1,n_2}^\lambda} \calF f \big)\Big|
		\Big\rVert_{L^2}
		\leq
		C
		s_1 s_2
		\vnorm{f}_{L^2}.
	\end{align}
	For the oscillation seminorm, we set $k_0=\max\big\{k \in \NN: (N_k, N_k),
	(N_{k+1}, N_{k+1}) \in U_w^{00}\big\}$ and observe that by the lacunarity of $\calN$
	\[
		\tau^{k_0} \leq N_{k_0+1} \leq \min\{S_1, S_2\}.
	\]
	Thus, using \eqref{eq:43} we get
	\begin{align*}
		\Big\lVert\calO_{00}\Big(
		\sum_{\lambda \in \Lambda} \calF^{-1} \big(\ind{R_{n_1,n_2}^\lambda} \calF f \big):
		n_1, n_2 \in \NN_0\Big)
		\Big\rVert_{L^2}
		\leq
		C k_0 s_1 s_2
		\vnorm{f}_{L^2}
		\leq C(s_1 s_2)^2
		\vnorm{f}_{L^2}.
\end{align*}

\vspace*{1ex} \noindent
	{\bf The region $U_w^{11}$.} We shall exploit the rationality of $\Lambda$. Suppose that
	for each $\lambda \in \Lambda$ we are given a function $f_\lambda \in L^2\big(\RR^2\big)$.
	We are going to show that there is a constant $C>0$ such that
	\begin{equation}
		\label{eq:12}
		\Big\lVert
		\sup_{(n_1, n_2) \in U_w^{11}}
		\Big|
		\sum_{\lambda \in \Lambda}
		e^{2\pi i \sprod{\lambda}{x}}
		\calF^{-1} \big(\ind{R_{n_1,n_2}} \calF f_\lambda \big)(x)
		\Big|
		\Big\rVert_{L^2({\rm d}x)}^2
		\leq
		C
		\sum_{\lambda \in \Lambda}
		\vnorm{f_\lambda}_{L^2}^2.
	\end{equation}
	For $x, y \in \RR^2$ we set
	\[
		I(x,y) =
		\sup_{(n_1, n_2) \in U_w^{11}}
		\Big|\sum_{\lambda \in \Lambda} e^{2\pi i \sprod{\lambda}{x}}
		\calF^{-1}\big(\ind{R_{n_1,n_2}} \calF f_\lambda \big)(y) \Big|,
	\]
	and
	\[
		J(x, y) =
		\sum_{\lambda \in \Lambda} e^{2\pi i \sprod{\lambda}{x}} f_\lambda(y).
	\]
	Let us observe that for each $y \in \RR^2$ the functions $x \mapsto I(x, y)$ and
	$x \mapsto J(x, y)$ are $(Q_1, Q_2)$-periodic. By Plancherel's theorem, for
	$u \in [0, Q_1] \times [0,Q_2]$ and $\lambda \in \Lambda$ we may estimate
	\begin{multline*}
		\big\lVert
		\calF^{-1}\big(\ind{R_{n_1,n_2}} \calF f_\lambda \big)(x)
		-
		\calF^{-1}\big(\ind{R_{n_1,n_2}} \calF f_\lambda \big)(x + u)
		\big\rVert_{L^2({\rm d}x)} \\
		=
		\big\lVert
		\ind{R_{n_1,n_2}}(\xi) \ind{R_{S_1,S_2}}(\xi)\calF f_{\lambda}(\xi)
		\big(1 - e^{2\pi i \sprod{\xi}{u}}\big)
		\big\rVert_{L^2({\rm d}\xi)}
		\lesssim
		\big(2^{-S_1} \abs{u_1} + 2^{-S_2} \abs{u_2}\big)
		\cdot \lVert f_\lambda \rVert_{L^2}.
	\end{multline*}
	Therefore, by the triangle inequality, Theorem \ref{th:2} and the Cauchy--Schwarz
	inequality
	\begin{multline*}
		\big\lvert
		\lVert I(x, x) \rVert_{L^2({\rm d}x)}
		-
		\lVert I(x, x + u) \rVert_{L^2({\rm d}x)}
		\big\rvert\\
		\leq
		\sum_{\lambda \in \Lambda}
		\Big\lVert
		\sup_{(n_1, n_2) \in U_w^{11}}
		\Big|
		\calF^{-1} \big( \ind{R_{n_1,n_2}} \calF f_{\lambda} \big)(x)
		-
		\calF^{-1} \big( \ind{R_{n_1,n_2}} \calF f_{\lambda} \big)(x+u)
		\Big|
		\Big\rVert_{L^2({\rm d}x)}\\
		\lesssim
		\sum_{\lambda \in \Lambda}
		\big\lVert
		\ind{R_{S_1, S_2}} (\xi) \calF f_{\lambda}(\xi) \big(1 - e^{2\pi i \sprod{\xi}{u}}\big)
		\big\rVert_{L^2({\rm d}\xi)} \\
		\lesssim
		\big(Q_1 2^{-S_1} + Q_2 2^{-S_2} \big) \sqrt{\abs{\Lambda}}
		\Big(\sum_{\lambda \in \Lambda} \lVert f_\lambda \rVert_{L^2}^2 \Big)^{1/2}.
	\end{multline*}
	Since $\sqrt{\abs{\Lambda}}\big(Q_1 2^{-S_1} + Q_2 2^{-S_2}\big) \leq 2$ we get
	\[
		\lVert I(x, x)
		\rVert_{L^2({\rm d}x)}
		\lesssim
		\lVert I(x, x + u) \rVert_{L^2({\rm d}x)}
		+
		\Big(
		\sum_{\lambda \in \Lambda} \lVert f_\lambda \rVert^2_{L^2}
		\Big)^{1/2}.
	\]
	By repeated change of variables and periodicity we obtain
	\begin{align*}
		\int_{[0, Q_1]\times[0,Q_2]}
		\int_{\RR^2} I(x, x+u)^2 {\: \rm d}x {\: \rm d}u
		=&
		\int_{[0, Q_1]\times[0,Q_2]}
		\int_{\RR^2} I(x-u, x)^2 {\: \rm d}x {\: \rm d}u\\
		=&
		\int_{\RR^2}
		\int_{[x_1-Q_1, x_1]\times[x_2-Q_2,x_2]}
		I(u, x)^2 {\: \rm d}u {\: \rm d}x\\
		=&
		\int_{\RR^2}
		\int_{[0, Q_1]\times[0,Q_2]}
		I(u, x)^2 {\: \rm d}u {\: \rm d}x\\
		=&
		\int_{[0, Q_1]\times[0,Q_2]}
		\int_{\RR^2} I(u, x)^2 {\: \rm d}x {\: \rm d}u.
	\end{align*}
	Next, for any $u \in [0, Q_1] \times [0, Q_2]$, by Theorem \ref{th:2} we get
	\[
		\big\lVert I(u, x) \big\rVert_{L^2({\rm d}x)}
		=
		\big\lVert
		\sup_{(n_1, n_2) \in U_w^{11}}
		\big|
		\calF^{-1}\big(\ind{R_{n_1, n_2}} J(u,\: \widehat{ \cdot } \:) \big)(x)
		\big|
		\big\rVert_{L^2({\rm d}x)}
		\leq
		C
		\big\lVert
		J(u, x)
		\big\rVert_{L^2({\rm d}x)}.
	\]
	By the orthogonality between exponential functions we obtain
	\begin{multline*}
        \int_{[0,Q_1]\times[0, Q_2]}
		\int_{\RR^2} \abs{J(u, x)}^2 {\: \rm d}x {\: \rm d}u
		=
		\int_{\RR^2}
		\int_{[0, Q_1] \times [0, Q_2]}
		\Big\lvert
		\sum_{\lambda \in \Lambda} e^{2\pi i \lambda u} f_\lambda(x)
		\Big\rvert^2
		{\rm d} u {\: \rm d} x\\
		=
		\sum_{\lambda_1, \lambda_2 \in \Lambda}
		\langle f_{\lambda_1} , f_{\lambda_2} \rangle
		\int_{[0, Q_1] \times [0, Q_2]}
		e^{2\pi i \sprod{(\lambda_1 - \lambda_2)}{u}}
		{\: \rm d} u
		=
		Q_1 Q_2
		\sum_{\lambda \in \Lambda} \vnorm{f_\lambda}^2_{L^2}.
	\end{multline*}
	Hence,
	\[
		\int_{[0,Q_1]\times[0,Q_2]}
		\vnorm{I(u, x)}_{L^2({\rm d} x)}^2 {\: \rm d} u
		\lesssim
		Q_1 Q_2
		\sum_{\lambda \in \Lambda} \vnorm{f_\lambda}_{L^2}^2
	\]
	and the proof of the claim \eqref{eq:12} is finished.

	To estimate the case with the supremum norm, we apply \eqref{eq:12} to
	$$
	\calF f_\lambda(\xi) = \ind{R_{0,0}}(\xi) \calF f(\xi + \lambda).
	$$
	Analogous reasoning gives the proof for the oscillation
        seminorm. One only needs to replace the supremum by the oscillation
        seminorm and use
	Theorem \ref{th:4} instead of Theorem \ref{th:2}.

	\vspace*{1ex}\noindent
	{\bf The region $U_w^{01}$.} We start by observing that
	\begin{align*}
		\Big\lVert\calO_{01}\Big(
		\seq{\sum_{\lambda \in \Lambda}
		\calF^{-1} \big(\ind{R_{n_1,n_2}^\lambda} \calF f \big)}
		{n_1, n_2 \in \NN_0}\Big)
		\Big\rVert_{L^2}
		\lesssim
		s_1
		\Big\|
		\sup_{(n_1, n_2) \in U_w^{01}}
		\Big|\sum_{\lambda \in \Lambda}
		\calF^{-1} \big(\ind{R_{n_1,n_2}^\lambda} \calF f \big)
		\Big|\Big\|_{L^2}
	\end{align*}
	since, thanks to lacunarity of $\calN$, the cardinality of the
        set $\{k \in \NN : N_{k+1} \leq S_1\}$ is
	bounded by a constant multiple of $s_1$. The proof will be completed if we show
	\[
		\Big\|
		\sup_{(n_1, n_2) \in U_w^{01}}
		\Big|
		\sum_{\lambda \in \Lambda}
		\calF^{-1} \big(\ind{R_{n_1,n_2}^\lambda} \calF f \big)
		\Big|
		\Big\|_{L^2}
		\lesssim s_1 \vnorm{f}_{L^2}.
	\]
	First, we apply \eqref{eq:6} to the sequence $\big(b_{n_1} : 1 \leq n_1 \leq S_1\big)$
	where
	\[
		b_{n_1} =\sup_{n_2\ge S_2}
		\Big|\sum_{\lambda \in \Lambda}
		\calF^{-1} \big(\ind{R_{n_1,n_2}^\lambda} \calF f \big)\Big|.
	\]
	In view of the proof for the region $U_w^{11}$   we get $\|b_{S_1}\|_{L^2} \lesssim \|f\|_{L^2}$.
	Therefore, it suffices to show that for every fixed $i \in \{0, \ldots, s_1\}$
	\begin{align}
    	\label{eq:24}
        \sum_{j = 1}^{2^{s_1-i}}
		\Big\|\sup_{n_2\ge S_2}
		\Big| \sum_{\lambda \in \Lambda}
		\calF^{-1} \Big(
		\big( \ind{R_{j2^{i},n_2}^\lambda} - \ind{R_{(j-1)2^{i},n_2}^\lambda} \big)
		\calF f
		\Big)\Big|
	 	\Big\|_{L^2}^2\lesssim\|f\|_{L^2}^2,
	\end{align}
	since
	\begin{align*}
		\big|b_{j2^i}-b_{(j-1)2^i}\big|
		\le
		\sup_{n_2 \ge S_2}
		\Big|
		\sum_{\lambda \in \Lambda}
		\calF^{-1} \Big(
		\big( \ind{R_{j2^{i},n_2}^\lambda} - \ind{R_{(j-1)2^{i},n_2}^\lambda} \big)
		\calF f
		\Big)\Big|.
	\end{align*}
	For a fixed $j \in \big\{1, \ldots, 2^{s_1-i}\big\}$ we may write
	\[
		\sum_{\lambda \in \Lambda}
		\Big( \ind{R_{j2^{i},n_2}^\lambda} - \ind{R_{(j-1)2^{i},n_2}^\lambda} \Big)
		=
		\sum_{\eta \in \Lambda_2}
		\sum_{\atop{\lambda \in \Lambda}{\lambda_2 = \eta}}
		\big(\ind{A_{j2^{i}}^{\lambda_1}} - \ind{A_{(j-1)2^{i}}^{\lambda_1}}\big)
		\otimes
		\ind{A_{n_2}^{\eta}},
	\]
	where $\Lambda_2 = \big\{\lambda_2 \in \RR : \lambda=(\lambda_1, \lambda_2)
	\in \Lambda\big\}$ and $A_n^\lambda = \lambda + A_n$. Therefore
    \begin{multline*}
		\sum_{\lambda \in \Lambda}
		\int_{\RR^2} e^{2\pi i \sprod{\xi}{x}}
		\Big(\ind{R_{j2^{i},n_2}^\lambda}(\xi) - \ind{R_{(j-1)2^{i},n_2}^\lambda}(\xi)\Big)
		\calF f(\xi) {\: \rm d}\xi\\
		=
		\sum_{\eta \in \Lambda_2}
		e^{2\pi i \eta x_2}
		\int_\RR e^{2\pi i \xi_2 x_2}
		\ind{A_{n_2}}(\xi_2)
		F_\eta(x_1, \widehat{\xi_2 + \eta})
		{\: \rm d} \xi_2,
    \end{multline*}
	where we have set
	\[
		F_\eta(x_1, x_2) =
		\sum_{\atop{\lambda \in \Lambda}{\lambda_2 = \eta}}
		\int_\RR
		e^{2\pi i \xi_1 x_1}
		\Big(\ind{A_{j2^{i}}^{\lambda_1}}(\xi_1) - \ind{A_{(j-1)2^{i}}^{\lambda_1}}(\xi_1)\Big)
		f(\widehat{\xi_1}, x_2)
		{\: \rm d} \xi_1.
	\]
	We notice that for any $\eta, \eta' \in \Lambda_2$, if $\eta \neq \eta'$ then
	\[
		\abs{\eta - \eta'} \geq Q_2^{-1}.
	\]
	Since the supremum is taken over $n_2 \geq S_2$ we may apply to $Q_2 \cdot \Lambda_2$ the
	one-parameter inequality from \eqref{eq:12} to obtain
	\begin{multline*}
		\Big\lVert
		\sup_{n_2 \ge S_2}
		\Big|\sum_{\eta \in \Lambda_2}
		e^{2\pi i \eta x_2}
		\int_\RR e^{2\pi i \xi_2 x_2}
		\ind{A_{n_2}}(\xi_2)
		F_\eta(x_1, \widehat{\xi_2 + \eta})
		{\: \rm d} \xi_2
		\Big|
		\Big\rVert_{L^2({\rm d} x_2)}^2\\
		=
		Q_2^{-1}
		\Big\lVert
		\sup_{n_2 \ge \log_2 \sqrt{\abs{\Lambda}}}
		\Big|\sum_{\eta \in Q_2 \Lambda_2}
		e^{2\pi i \eta x_2}
		\int_\RR e^{2\pi i \xi_2 x_2}
		\ind{B_{n_2}}(\xi_2)
		F_{Q_2^{-1} \eta}\big(x_1, \widehat{(Q_2^{-1}\xi_2 + Q_2^{-1}\eta)}\big)
		{\: \rm d} \xi_2
		\Big|
		\Big\rVert_{L^2({\rm d} x_2)}^2\\
		\lesssim
		Q_2^{-1}
		\sum_{\eta \in Q_2 \Lambda_2}
		\int_\RR
		\Big\lvert
		\ind{B_0}(\xi_2)
		F_{Q_2^{-1} \eta}\big(x_1, \widehat{(Q_2^{-1}\xi_2 + Q_2^{-1}\eta)}\big)
		\Big\rvert^2
		{\: \rm d} \xi_2\\
		=
		\sum_{\eta \in \Lambda_2}
		\int_\RR
		\Big\lvert
		\ind{A_{j_0}}(\xi_2)
		F_{\eta}(x_1, \widehat{\xi_2 + \eta})
		\Big\rvert^2
		{\: \rm d} \xi_2
	\end{multline*}
	where $B_n = \big(-2^{-n-1} Q_2 2^{-j_0}, 2^{-n-1} Q_2 2^{-j_0}\big)$ and $j_0 = \lfloor \log_2 Q_2\rfloor$.
	Next, $\eta \in \Lambda_2$, by Plancherel's theorem, applied
        with respect to the first variable, we get
	\begin{multline*}
		\int_\RR \int_\RR
		\big\lvert
		\ind{A_{j_0}^\eta}(\xi_2)
		F_\eta(x_1, \widehat{\xi_2})
		\big\rvert^2
		{\: \rm d} x_1 {\: \rm d} \xi_2
		=
		\int_\RR
		\int_\RR
		\big\lvert
		\ind{A_{j_0}^\eta}(\xi_2)
		\calF F_\eta(\xi_1, \xi_2)
		\big\rvert^2
		{\: \rm d} \xi_1 {\: \rm d} \xi_2 \\
		=
		\int_\RR
		\int_\RR
		\Big\lvert
		\sum_{\atop{\lambda \in \Lambda}{\lambda_2 = \eta}}
		\Big(\ind{A_{j2^{i}}^{\lambda_1}}(\xi_1) - \ind{A_{(j-1)2^{i}}^{\lambda_1}}(\xi_1)\Big)
		\ind{A_{j_0}^\eta}(\xi_2)
		\calF f(\xi)
		\Big\rvert^2
		{\: \rm d}\xi_1 {\: \rm d} \xi_2.
	\end{multline*}
	Moreover,
	$$
	\sup_{\xi_1 \in \RR}
	\sum_{j = 1}^{2^{s_1-i}}
	\Big\lvert
	\sum_{\atop{\lambda \in \Lambda}{\lambda_2 = \eta}}
	\ind{A_{j2^{i}}^{\lambda_1}}(\xi_1) - \ind{A_{(j-1)2^{i}}^{\lambda_1}}(\xi_1)
	\Big\rvert
	\leq
	1,
	$$
	since for any $\lambda, \lambda' \in \Lambda$, if $\lambda \neq \lambda'$
	and $\lambda_2 = {\lambda_2'} =  \eta$ then $\abs{\lambda_1 - {\lambda_1'}} \geq 1$.
	Therefore,
	\begin{multline*}
		\sum_{j=1}^{2^{s_1-i}}
		\sum_{\eta \in \Lambda_2}
		\int_\RR
		\int_\RR
		\Big\lvert
		\ind{A_{j_0}}(\xi_2)
		F_\eta(x_1, \widehat{\xi_2 + \eta})
		\Big\rvert^2
		{\: \rm d} x_1
		{\: \rm d} \xi_2\\
		=
		\sum_{j=1}^{2^{s_1-i}}
		\sum_{\eta \in \Lambda_2}
		\int_\RR
		\int_\RR
		\Big\lvert
		\sum_{\atop{\lambda \in \Lambda}{\lambda_2 = \eta}}
		\Big(\ind{A_{j2^{i}}^{\lambda_1}}(\xi_1) - \ind{A_{(j-1)2^{i}}^{\lambda_1}}(\xi_1)\Big)
		\ind{A_{j_0}^\eta}(\xi_2)
		\calF f(\xi)
		\Big\rvert^2
		{\: \rm d}\xi_1 {\: \rm d} \xi_2\\
		\lesssim
		\sum_{\eta \in \Lambda_2}
		\int_\RR
		\int_\RR
		\big\lvert
		\ind{A_{j_0}^\eta}(\xi_2)
		\calF f(\xi)
		\big\rvert^2
		{\: \rm d}\xi_1 {\: \rm d}\xi_2
		\lesssim
		\vnorm{f}_{L^2}^2
	\end{multline*}
	which shows \eqref{eq:24} and concludes the proof of Theorem \ref{th:3}.
\end{proof}

\begin{bibliography}{discrete}
	\bibliographystyle{amsplain}
\end{bibliography}

\end{document}